\documentclass{article}[10pt]

  \usepackage{amsmath}
  \usepackage{graphicx}
  \usepackage{epsfig}
  \usepackage[latin1]{inputenc}
  \usepackage[english]{babel}
  \usepackage{amssymb}
  \usepackage{hyperref}
  \usepackage{theorem}
  \usepackage{xypic}

\newtheorem{theorem}{Theorem}

\newtheorem{lemma}[theorem]{Lemma}
\newtheorem{proposition}[theorem]{Proposition}

{ \theorembodyfont{\rmfamily}
\newtheorem{definition}{Definition}
\newtheorem{remark}{Remark}
\newtheorem{example}{Example}
}

\newenvironment{proof}
  {\begin{trivlist}
  \item[\textit{\noindent\textsc{ Proof.}}]}
  {\hfill$\square$\end{trivlist}}

\begin{document}

\title{Generalized Linear Cellular Automata in Groups and Difference Galois Theory}
\author{David Bl\'azquez Sanz\footnote{Corresponding author}
\footnote{
Universidad Nacional de Colombia - Sede Medell\'in. 
Calle 59A No. 63 - 20.
Medell\'in - Antioquia - Colombia. Ph. (57)(4)4309877. Fax. (57)(4)4309000.
e-mail: dblazquezs@unal.edu.co
} 
\& Weimar Mu\~noz\footnote{
e-mail: weimar.munoz@unimilitar.edu.co
}
}
\date{}

\maketitle

\begin{abstract}
Generalized non-autonomous linear celullar automata 
are systems of linear difference equations with many variables
that can be seen as convolution equations in a discrete group.
We study those systems from the stand point
of the Galois theory of difference equations and discrete Fourier transform.
\end{abstract}


\smallskip

\section*{Introduction}
  
    Natural science gives us many examples of dynamical systems in which 
the state variables interact in an ordered and homogeneous way. This kind of
systems have been recently named \emph{coupled cell network systems}. They
admit discrete groups of symmetries, and this fact leads to 
some interesting results concerning synchronization and
bifurcation theory, see for instance \cite{Go_etal}.

	In this article we study certain class of coupled cell network systems.
Namely those whose underlying network can be seen as a group (truly, a torsor).
These are the so-called \emph{generalized cellular automata in groups} \cite{Ce_etal}.
Inside this class there is the subclass of systems in which 
the transition function is a linear function. These are called 
\emph{generalized linear cellular automata in groups}, and they
are the main focus of this article. 

  First section is devoted to the presentation of the objects and some generalities
on their spaces of solutions. In second section we study 
explicit formulas for the solutions 
obtained by means of discrete Fourier transformation. 
We explain in detail the case of finite groups, which is reduced 
by means of Peter-Weil theorem and then we give some results concerning 
the infinite case. Theorems \ref{theorem25} and \ref{theorem26}
clarify the structure of the solutions with finite support and the sequences of numbers
appearing therein. Theorem \ref{theorem29} details the sequences of numbers that may
appear in periodic solutions. Third section is devoted to
Galois theory. Theorem \ref{theorem35} interprets the Fourier transform,
in the finite case, as a reduction of the structural group of the linear
difference system. Then, in order to apply difference Galois
theory to cellular automata in infinite groups we have to introduce some kind
of extensions of infinite transcendent degree. Theorem \ref{theorem39}
explains the structure of such extensions, being their Galois groups 
pro-algebraic groups.

\section{Generalized cellular automata in Cayley graphs}

\subsection{Cayley graphs} The configuration space of a cellular automaton 
is the space of states -complex valued functions- of a network. This network
is a directed graph (digraph) with a group of symmetries acting free and
transitively on its nodes. This notion is captured by the definition of Cayley
graph.  

\begin{definition}
Let $G$ be a discrete group and $U = \{u_1,\ldots,u_n\}$ be a finite subset of $G$. The
cayley graph $\mathcal G = (G,U)$ is the digraph satisfying:
\begin{enumerate}
\item[(a)] The set of nodes of $\mathcal G$ is $G$.
\item[(b)] There is an arrow from $g_1$ to $g_2$ if and only if $g_1g_2^{-1}\in U$.
\end{enumerate} 
\end{definition}

\begin{figure}[htb]
\centerline{\includegraphics[height=6cm]{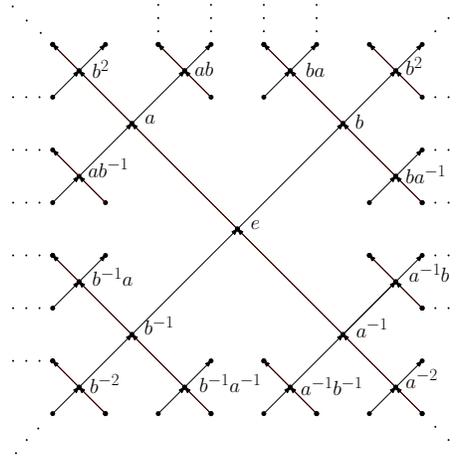}}
\caption{Cayley graph in a a free group with two generators $(\langle a, b \rangle, \{a, b\})$. 
}\label{fig1}
\end{figure}

\begin{remark} Some elementary facts about Cayley graphs:
\begin{itemize}
\item[(i)] The set of out-neighbours of $g\in G$ is $U^{-1}g$. 
\item[(ii)] The set of in-neighbours of $g\in G$  is $gU$.
\item[(iii)] $\mathcal G$ is an homogeneous graph of in-degree and out-degree equal to $n$, the cardinal of $U$.
\item[(iv)] By composition on the right side $G$ acts as a group of automorphisms of $\mathcal G$.
\item[(v)] By composition on the left side the centralizer of $U$ acts a as group of automorphisms of $\mathcal G$.
\item[(vi)] A graph is isomorphic to some Cayley graph if and only if it admits a group
of automorphisms acting transitively and freely on the set of nodes.
\end{itemize}
\end{remark}

We denote by $\mathcal C(G,\mathbf C)$ the  set of functions from $G$ to the field
of complex numbers, and by $\mathcal C_c(G,\mathbf C)$ the set
of complex functions with finite support. 
Note that $\mathcal C_c(G,\mathbf C)$ is isomorphic to the enveloping 
algebra $\mathbf C[G]$, and coincides with $\mathcal C(G, \mathbf C)$ for finite $G$. 
For $s\in\mathcal C(G,\mathbf C)$ and $g\in G$ we will denote by $s_g$ or $s(g)$ 
the value of $s$ at $g$, and we will term it the \emph{state of $s$ at $g$}.

The group $G$ acts in $\mathcal C(G, \mathbf C)$ naturally on the left and right side. In order to 
keep coherence with the classical convolution operators we will use the following notation:
\begin{eqnarray}
 (g\star s)_h &\colon =& s_{g^{-1}h}.  \label{G_action_L}\\
 (s\star g)_h &\colon =& s_{hg^{-1}}.  \label{G_action_R}
\end{eqnarray}

From now on let us fix a Cayley graph $\mathcal G = (G,U)$ with $U = \{u_1,\ldots,u_n\}$.

\begin{definition}\label{def_celA}
A generalized cellular automaton in $\mathcal G$ is a map $\varphi$ from $\mathcal C(G,\mathbf C)$ to itself satisfying:
\begin{itemize}
\item[(a)] It commutes with the right action of $G$, $\varphi(s\star g) = \varphi(s)\star g$.
\item[(b)] The state $\varphi(s)_g$ depends only of the states of $s$ at the out-neighbours of $g$,
$s_{u_1^{-1}g},\ldots,s_{u_n^{-1}g}$.
\end{itemize}
\end{definition}

\begin{remark}
It is easy to see that definition \ref{def_celA} 
coincides with the classical notion of cellular automata when
the group $G$ is a regular lattice $\mathbf Z^d$ or a quotient of this. 
In such case, the dimension
of the automaton is $d$ and its rank is the radius of the set $U$. 
\end{remark}

\begin{figure}[htb]
\centerline{\includegraphics[height=1cm]{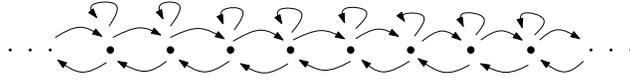}}
\caption{Cayley graph $(\mathbf Z, \{-1, 0, 1\})$ in a the additive group of integer numbers. Generalized 
cellular automata in this graph with states in $\mathbf C$ naturally correspond to classical rank one, one-dimensional
cellular automata with states in $\mathbf C$.
}\label{fig2}
\end{figure}

From conditions (a) and (b) it follows that $\varphi$ is determined by 
a map $f\colon {\mathbf C}^n\to {\mathbf C}$ satisfying:
$$\varphi(s)_g = f(s_{u_1^{-1}g},\ldots,s_{u_n^{-1}g}).$$
This map $f$ is called the \emph{transition map} of $\varphi$.
The automaton $\varphi$ is said to be polynomial if the transition map $f$
is a polynomial map, and it is said to be \emph{linear}
if the transition map $f$ is a linear map. 

From now on we will only consider cellular automata with transition functions
$f$ satisfying $f(0,\ldots,0) = 0$. This condition implies that the map 
$\varphi$ preserves the subspace $\mathcal C_c(G,\mathbf C)$ and thus it makes sense
to study the evolution of the automaton with finite support states.

\begin{figure}[htb]
\centerline{\includegraphics[height=5cm]{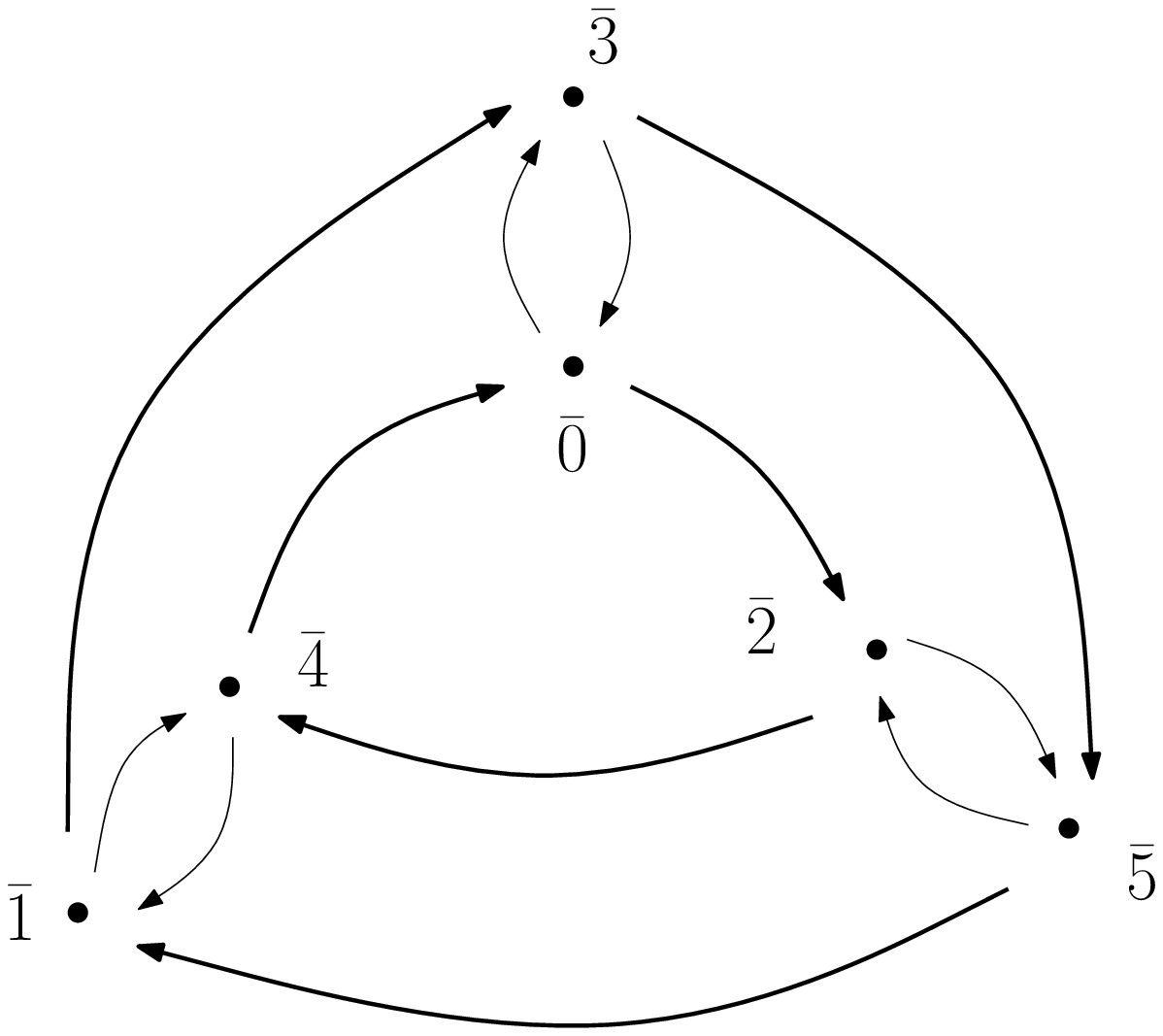}\hspace{1cm}\includegraphics[height=5cm]{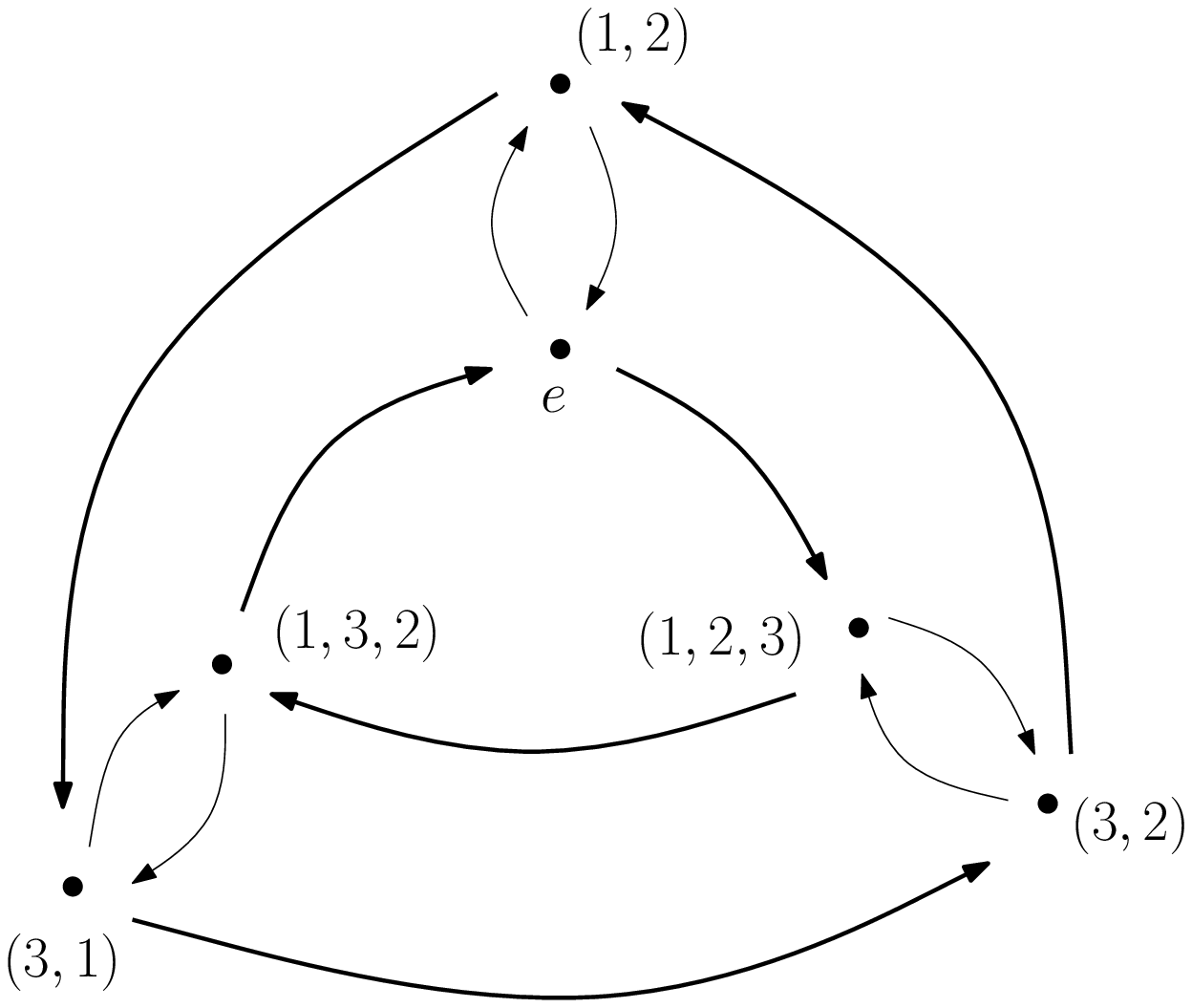}}
\caption{Left, Cayley graph in a cyclic group $(\mathbf Z_6, \{\bar 4, \bar 3\})$. 
Right, Cayley graph in the symmetric group $(S_3,\{(1,3,2),(1,2)\})$. 
}\label{fig3}
\end{figure}

We are interested in determining the dynamics of the cellular automata. 
This is, to give a description
of the sequences of states $(s^{(t)})_{t\in\mathbf Z_+}$ defined by the recurrence:
\begin{equation}\label{aut_caut}
s^{(t+1)} = \varphi(s^{(t)}).
\end{equation}

\subsection{Non-autonomous generalized cellular automata} The difference
equation \eqref{aut_caut} is autonomous in the following sense: the
cellular automata $\varphi$ does not depend explicitly on the \emph{step
variable} $t$. The theory shown in this article allows us to consider
a broader class of recurrences of the form: 
\begin{equation}\label{naut_caut}
s^{(t+1)} = \varphi(s^{(t)},t).
\end{equation}

Let $\rm Seq(\mathbf C)$ be the ring of sequences in $\mathbf C$.
This ring has a natural structure of $\sigma$-ring - also called difference ring - 
with the shift operator $\sigma(x^{(t)}) = x^{(t+1)}$. 
The space $\mathcal C(G,\rm Seq(\mathbf C))$ is the
space of bivariate discrete functions from $G\times\mathbf Z_+$ to $\mathbf C$.
Within this space we denote by $\mathcal C_c(G,\rm Seq(\mathbf C))$ to the
space of functions whose support is \emph{finite for any given $t$}.
The shift operator $\sigma$ naturally extends to $\mathcal C(G,\rm Seq(\mathbf C))$
and $\mathcal C_c(G,\rm Seq(\mathbf C))$ by letting $\sigma(s)_g^{(t)} = s_g^{(t+1)}$.
It is clear that a sequence of states $s \in\mathcal C(G,\rm Seq(\mathbf C))$ is 
solution of the difference equation \eqref{aut_caut} if an only
if it satisfies the functional equation:
\begin{equation}\label{sigma_phi}
\sigma(s) = \varphi(s).
\end{equation}

A first definition of a non-autonomous generalized cellular automata
would be an arbitrary sequence of transition laws $f(\bullet,t)$, defining
automata $\varphi(\bullet,t)$ and functional equation \eqref{naut_caut}.
However, this definition is too broad, since we are allowing any kind of
sequence of transition laws. If we want to explore the link between
the sequence $f(\bullet,t)$ of transition laws and the solutions
of \eqref{naut_caut} we should not allow arbitrary sequences.

In order to see this link, 
we fix ${\mathbf k}$, a $\sigma$-subring of  ${\rm Seq}(\mathbf C)$
whose field of constants is ${\mathbf k}^\sigma = \mathbf C$. The $\sigma$-ring
$\mathbf k$ is the ring of allowed sequences in the definition of our
cellular automata. 


\begin{definition}
Let $f\colon {\mathbf k}^n\to {\mathbf k}$ be polynomial map satisfying $f(0,\dots,0) = 0$, and let: 
$$\varphi\colon\mathcal C(G,{\mathbf k})\to\mathcal C(G,{\mathbf k}),$$
defined by:
$$\varphi(s)_g = f(s_{u_1^{-1}g},\ldots,s_{u_n^{-1}g}).$$
We say that $\varphi$ is a non-autonomous generalized polynomial cellular automaton  in $\mathcal G$ with
coefficients in ${\mathbf k}$.
\end{definition}

Note that, for ${\mathbf k} = \mathbf C$ we obtain autonomous 
polynomial cellular automata. Note also that
the transition map $f$ naturally extends to any $\sigma$-extension of ${\mathbf k}$, in particular
to $f\colon {\rm Seq}(\mathbf C)^n\to {\rm Seq}(\mathbf C)$. A solution of 
the cellular automata $\varphi$ is $s\in \mathcal C(G,{\rm Seq}(\mathbf C))$ 
satisfying \eqref{sigma_phi}. To study the automaton $\varphi$ 
from a difference algebra standpoint
is to study the intermediate $\sigma$-extensions ${\mathbf k}\subseteq R \subseteq {\rm Seq}(\mathbf C)$ 
such that the automaton has some solutions $s\in \mathcal C(G,R)$. Thus, we may 
know the kind of sequences giving rise to solutions of the automaton. There is at least
three $\sigma$-rings canonically attached to the automata $s$:
\begin{itemize}
\item[(a)] The $\sigma$-ring ${\mathbf k}\subseteq L$ spanned by 
all the solutions of the automaton. 
It is the smallest $\sigma$-ring extension $L$ of ${\mathbf k}$ in 
${\rm Seq}(\mathbf C))$ such
that all the solutions of the automaton are in $\mathcal C(G,L)$
\item[(b)] The $\sigma$-ring ${\mathbf k}\subseteq L^{c}$ spanned by all 
the solutions of finite support.
It is the smallest $\sigma$-ring extension $L^c$ of ${\mathbf k}$ in $L$ such 
that all the solutions
in $\mathcal C_c(G,{\rm Seq}(\mathbf C))$ are in $\mathcal C_c(G,L^c)$.
\item[(c)] Given a subgroup $H\subset G$, a sequence of states 
$s\in\mathcal C(G,{\rm Seq}(\mathbf C))$ is called $H$-periodic if $s_{g} = s_{gh}$ for 
all $g\in G$ and $h\in H$. We consider the $\sigma$-ring $L^{H-\rm per}$ spanned by all the
$H$-periodic solutions, and the $\sigma$-ring $L^{\rm per}$ spanned by all the 
$H$-periodic solutions for all normal subgroups of finite index of $G$.
\end{itemize}

It is clear that $L_c\subseteq L$ and $L^{\rm per}\subseteq L$. For finite group $G$ these three
$\sigma$-extensions are the same. For infinite groups $L$ tends to be the whole ring of sequences.
The following result holds.

\begin{proposition}
Let us consider the topology of the coincidence along compact (and thus finite) subsets of
$\mathbf Z_+$ in ${\rm Seq}(\mathbf C)$. 
The following statements are satisfied:
\begin{itemize}
\item[(a)] The ring $L^c$ is dense in $L$. 
\item[(b)] The ring $L^{\rm per}$ is dense in $L$ if and only if $G$ is a residually finite group.
\end{itemize}
\end{proposition}

\begin{proof}
(a) The ring $L$ is spanned by the sequences $x^{(t)} = s_e^{(t)}$ 
for solutions $s$ of \eqref{naut_caut} where 
$s^{(0)}\in\mathcal C(G,\mathbf C)$ variates among all the possible initial
data. Any initial data $s^{(0)}$ can be approximated by finite support initial data.
(b) Same consideration, taking into account that any initial data can be approximated
by periodic initial data if and only if $G$ is residually finite.
\end{proof}

\subsection{Linear generalized non-autonomous cellular automata}
The classical convolution operator for functions from $G$ to  
$\mathbf C$ naturally extends to functions in the group with values in any subring $R\subseteq {\rm Seq}(\mathbf C)$.
$$\star\colon \mathcal C(G,R) \times_c \mathcal C(G,R)\to  \mathcal C(G,R),\quad
 (\alpha\star \beta)_g = \sum_{h\in G} \alpha_{h}\beta_{h^{-1}g}.$$
Here, the set  $\mathcal C_c(G,R) \times_c \mathcal C(G,R) \subset \mathcal C(G,R) \times_c \mathcal C(G,R)$
is just the maximal subset of $\mathcal C_c(G,R) \times \mathcal C(G,R)$ 
in which the convolution is defined. It is clear that
$\alpha$ or $\beta$ with finite support is a sufficient condition for the existence of the
convolution, and that the convolution is a $R$-linear associative operation whenever it is 
defined. 

Let us denote by $\delta_g$ the Dirac delta function at the point $g\in G$. We
have the following formulae relating the convolution and the action of $G$:
\begin{eqnarray}
\delta_g \star \beta = g \star \beta \\
\alpha \star \delta_g = \alpha \star g.
\end{eqnarray}

\begin{definition}
A \emph{non-autonomous generalized linear cellular automaton in the Cayley graph 
$\mathcal G$ with coefficients in ${\mathbf k}$} is a 
non-autonomous generalized polynomial cellular automata whose transition function is a linear homogeneous polynomial:
$$\varphi(s)_g = a_1s_{u_1^{-1}g} + \ldots + a_ns_{u_n^{-1}g}, \quad a_i\in\mathbf k.$$ 
\end{definition}
Therefore the non-autonomous generalized linear cellular automaton difference equation:
$$s_{g}^{(t+1)} = a_1^{(t)}s_{u_1^{-1}g} + \ldots + a_n^{(t)}s_{u_n^{-1}g}^{(t)}$$
can be seen as a difference convolution equation:
\begin{equation}\label{conv_eq}
\sigma(s) = \alpha \star s,
\end{equation}
where $\alpha\in \mathcal C_c(G,{\mathbf k})$ is the function from $G$ to $\mathbf k$
with support in $U$ that maps each $u_j\in U$ to the sequence $a_j\in {\mathbf k}$.

\subsection{Three solution spaces associated to a cellular automata} 
From now on, let us fix $\alpha\in \mathcal C_c(G,{\mathbf k})$ and consider the convolution equation 
\eqref{conv_eq}. For any intermediate $\sigma$-extension ${\mathbf k}\subseteq R \subseteq L$ let 
us denote ${\rm Sol}(\alpha,R)\subset \mathcal C(G,R)$ the space
of solutions of in $\mathcal C(G,R)$. It is clear that ${\rm Sol}(\alpha,R)$ is a 
vector space over $\mathbf C$.

By the associativity of the convolution, 
if $s$ is a solution of \eqref{conv_eq} and $g\in G$ then 
$s\star g = s \star \delta_{g}$ is also a solution. Thus,
\emph{the spaces ${\rm Sol}(\alpha,R)$ are right $G$-modules}. As before, we have some distinguished
subspaces of the total space of solutions.

\begin{itemize}
\item[(a)] ${\rm Sol}(\alpha,R)$ is the space of all solutions of \eqref{conv_eq} in $\mathcal C(G,R)$.
We have that ${\rm Sol}(\alpha,{\rm Seq}(\mathbf C)) = {\rm Sol}(\alpha,L)$.
\item[(b)] ${\rm Sol_c}(\alpha, R)$ is the space of all finite support solutions of \eqref{conv_eq}
in $\mathcal C_c(G,R)$. We have ${\rm Sol}_c(\alpha,{\rm Seq}(\mathbf C)) = {\rm Sol}_c(\alpha,L^c)$.
\item[(c)] Given $H\subset G$ let us denote  ${\rm Sol_{H-{\rm per}}}(\alpha, R)$ the space of al $H$-periodic
solutions, and ${\rm Sol_{\rm per}}(\alpha, R)$ the space of periodic $H$-solutions for 
all normal subgroups $H$ of finite index. We have ${\rm Sol}_{H-{\rm per}}(\alpha,{\rm Seq}(\mathbf C)) = {\rm Sol}_{H-{\rm per}}(\alpha,L^{H-{\rm per}})$
and ${\rm Sol}_{\rm per}(\alpha,{\rm Seq}(\mathbf C)) = {\rm Sol}_{\rm per}(\alpha,L^{\rm per})$.
\end{itemize}

\begin{definition}
We call fundamental solution $\phi$ of \eqref{conv_eq} to the only solution satisfying:
$$\phi^{(0)}_e = 1, \quad \phi^{(0)}_g = 0, \mbox{ for all } g\neq e.$$
Let $H\subset G$ be a subgroup, we call fundamental $H$-periodic solution $\phi^H$ of
\eqref{conv_eq} to the only solution satisfying:
$$\left(\phi^{H}\right)^{(0)}_h = 1, \mbox{ for all } h\in H,\quad 
\left(\phi^{H}\right)^{(0)}_g  = 0, \mbox{ for all } g \not\in H.$$
\end{definition}

Note that 
$\phi^{H} = \sum_{h\in H} \phi\star h,$
where this infinite sum of functions in $G\times \mathbf Z^+$ becomes a finite sum
after evaluation in $(g,t)\in G\times \mathbf Z^+$.

\subsection{Periodic solutions and quotient cellular automata} 

Let $H\lhd G$ be a normal subgroup.
Here we discuss how
the $H$-periodic solutions of \eqref{conv_eq} can be seen as the solutions
of a convolution equation in the quotient $G/H$.
 Let us denote by $\pi_H\colon G\to G/H$, $g\mapsto [g]$, the quotient
map and $\pi_H^*$ the natural embedding:
$$\pi^*_H\colon \mathcal C(G/H,R)\to \mathcal C(G,R), \quad \pi_H^*(f)(g) = f([g]).$$
Let us consider the sub-ring $\mathcal C_H(G,R)$ consisting in functions with finite
supports along co-sets:
$$\mathcal C_H(G,R) = \{f\in \mathcal C(G,R) | \,\forall g\in G\,\forall t\in\mathbf Z_+ \,  
|{\rm supp}(f^{(t)}) \cap gH|< \infty \}.$$
It is clear that $\mathcal C_H(G,R) \supseteq \mathcal C_c(G,R)$, and this inclusion
is an equality if and only if $H$ is of finite index in $G$. We define the map,
$$r_H\colon \mathcal C_H(G,R)\to \mathcal C(G/H,R), (r_Hf)([g]) = \sum_{h\in H} f(gh).$$

\begin{lemma}\label{adjoint}
The maps $\pi_H^*$ and $r_H$ are adjoint with respect to convolution in the 
following sense: let $\alpha$ be in $\mathcal C_c(G,R)$ and $s\in \mathcal C(G/H,R)$, then
$$\alpha\star \pi_H^*s = \pi_H^* (r_H\alpha \star s).$$
\end{lemma}

\begin{proof}
A direct computation yields:
$$(\alpha\star\pi_H^*s)_g = \sum_{h\in G}\alpha_h s_{[h^{-1}g]} = 
\sum_{[\tau]\in G/H}\sum_{h\in H} \alpha_{\tau h} s_{[h^{-1}\tau^{-1}g]}.$$
Using that $H$ is normal in $G$ for $h\in H$ we have 
$h^* = g^{-1}\tau h\tau^{-1}g\in H$ and thus 
$[h^{-1}\tau^{-1}g] = [h^{-1}\tau^{-1}gh^*] = [\tau^{-1}g]$. We gather the terms
in the above sum obtaining,
$$\sum_{[\tau]\in G/H}\left(\sum_{h\in H} \alpha_{\tau h} \right)s_{[\tau^{-1}g]}
= \sum_{[\tau]\in G/H} (r_H\alpha)_{[\tau]}s_{[\tau^{-1}g]} = (r_H\alpha\star s)_{[g]}.$$
\end{proof}

From lemma \ref{adjoint} is clear that $\pi_H^*s$ is a $H$-periodic
solution of \eqref{conv_eq} if and only if $s$ is a solution of the
convolution equation,
$$\sigma(s) = r_H \alpha \star s,$$
in $G/H$. We have proven the following result.

\begin{proposition}
The embedding $\pi_H^*\colon \mathcal C(G/H,R)\to \mathcal C(G,R)$ identifies
the space ${\rm Sol}(r_H\alpha, R)$ with ${\rm Sol}_{H-\rm per}(\alpha,R)$, for
any $\sigma$-ring $R$. In particular, the image by $\pi_H^*$ of the fundamental
solution in $G/H$ is the fundamental $H$-periodic solution in $G$. 
\end{proposition}

\begin{example} (Abel-Liouville theorem for cellular automata) The fundamental $G$-periodic solution can
be easily computed and yields:
$$\left(\phi^G\right)^{(t)}_g = \prod_{\tau=0}^{t-1}\sum_{i=1}^na_i^{(\tau)},$$
it follows that for any solution with finite support $s\in {\rm Sol}(\alpha, L^c)$ the following holds:
$$\sum_{g\in G} s^{(t)}_g =  \left(\prod_{\tau=0}^{t-1}\sum_{i=1}^na_i^{(\tau)}\right)\sum_{g\in G}s^{(0)}_g.$$
\end{example}

\subsection{Approximation of solutions by periodic and finite support solutions}
Here we give some remark about the topology of the space of solutions. 
In this section we consider in $\mathcal C(G,\rm Seq(\mathbf C))$ the
topology of the convergence along compact subsets of $\mathbf Z_+\times G$.

\begin{proposition}\label{prop_monog}
The following statements hold:
\begin{itemize}
\item[(a)] ${\rm Sol}_c(\alpha,L^c)$ is a monogenous $G$-module, namely:
$${\rm Sol}_c(\alpha,L^c) = \bigoplus_{g\in G} \langle \phi\star g \rangle.$$
\item[(b)] ${\rm Sol}(\alpha,L) = \prod_{g\in G}\langle \phi\star g \rangle.$ 
\item[(c)] ${\rm Sol}_{\rm per}(\alpha,L^{\rm per})$ is union of finite dimensional $G$-modules, and
admits a natural structure of $G^{\rm pro}$-module, where $G^{\rm pro}$ denotes the profinite completion
of $G$.
\end{itemize}
\end{proposition}

\begin{proof}
(c) Note that if $H\subset F$, then the space of ${F}$-periodic solutions 
is embedded into the space of $H$-periodic
solutions. This proves that ${\rm Sol}(\alpha,L^{\rm per})$ is union of finite dimensional
$G$-modules. Let $s$ be in ${\rm Sol}(\alpha,L^{\rm per})$. 
In particular $s$ is $H$-periodic for some $H\lhd G$
of finite index. By definition, an element $\tilde g\in G^{\rm pro}$ is of the form $([g_i])_{i\in I}$,
with $[g_i]\in G/H_i$, where $I$ indexes the set of normal subgroups of $G$ of finite index. 
In particular, the component of $\tilde g$ in $G/H$ is a certain class $[g]$ represented by an element
$g\in G$. The formula, $s\star\tilde g = s\star g$ extends the $G$-module structure of  ${\rm Sol}(\alpha,L^{\rm per})$
to a $G^{\rm pro}$-module structure.
\end{proof}

\begin{proposition}
The following statements hold:
\begin{itemize}
\item[(a)] ${\rm Sol}_c(\alpha,L^c)$ is dense in ${\rm Sol}(\alpha,L)$.
\item[(b)]${\rm Sol}_{\rm per}(\alpha,L^{\rm per})$ is dense in ${\rm Sol}(\alpha,L)$ if
and only if $G$ is residually finite.
\item[(c)] If $G$ is not finite ${\rm Sol}_c(\alpha,L^{c})\cap{\rm Sol}_{\rm per}(\alpha,L^{\rm per}) = \{0\}$.
\item[(d)] If $G$ is finite then 
${\rm Sol}(\alpha,L) = {\rm Sol}(\alpha,L^{c}) = {\rm Sol}(\alpha,L^{\rm per}) \simeq \mathbf C^{|G|}$.
\end{itemize}
\end{proposition}

\begin{proof}
(a) and (b). As before, it is clear that to approximate a solution $s_g^{(t)}$ 
it suffices to approximate initial conditions $s_{g}^{(0)}$
along arbitrary finite sets in $G$. 
The initial conditions can be approximated by compact support initial conditions in any case.
The initial conditions can be approximated by periodic initial conditions only if $G$ is residually finite.
(c) If $G$ is not finite for any non zero periodic solution the support for any given $t$
is infinite. (d) If $G$ is finite there is not difference between periodic, compact support or
arbitrary initial data. The space of solutions is identified with the space of initial conditions
which is $\mathcal C(G,\mathbf C) \simeq \mathbf C^{|G|}$.
\end{proof}

\section{Fourier transform and linear cellular automata.}

  In this section we show how to reduce 
convolution equation \eqref{conv_eq} to its simplest form. Fourier transform
diagonalizes convolution (see \cite{Terras} for the finite Fourier transform 
and Peter-Weil theorem and \cite{Rudin} for the discrete abelian case) 
and it is a $\mathbf C$-linear
operator. Thus, the Fourier transform of a convolution equation will be 
a simpler linear difference system.

Let us recall that the dual $\widehat{G}$ of $G$ is set of irreducible unitary 
representations of $G$, modulo $G$-isomorphisms. Two important facts are the 
following: the dual of a finite group is a finite set; the dual of discrete abelian group
is a compact abelian group.

\subsection{Finite groups}
Let us consider $G$ a finite group. We take representatives $\rho_1,\ldots,\rho_m$
of the irreducible unitary representations of $G$: 
$$\rho_{k}\colon G \to {\rm GL}(d_{k}, \mathbf C),$$
in such way that we can identify the $\widehat G$ with the set $\{\rho_1,\ldots,\rho_m\}$.
Let us denote by $\mathcal O(\hat G, \mathbf C)$ the space:
$$\mathcal O(\hat G, \mathbf C) =  \bigoplus_{k = 1}^m {\rm Mat}(d_{k}\times d_{k}, \mathbf C).$$ 
This $\mathcal O(\hat G, \mathbf C)$ is seen as a space of matrix valued functions 
defined in $\widehat{G}$. Each element $f\in \mathcal O(\hat G, \mathbf C)$
is a function that assigns a matrix 
$f(\rho_{k})\in  {\rm Mat}(d_{k}\times d_{k}, \mathbf C)$ to each
irreducible representation $\rho_{k}$. Peter-Weil theorem states that the Fourier transform in $G$
is a $\mathbf C$-linear isomorphism:
$$\hat\bullet\colon \mathcal C(G,\mathbf C)\xrightarrow{\sim} \mathcal O(\hat G, \mathbf C),\quad  
\hat f(\rho_{k}) = \sum_{{k}=1}^m f(g)\rho_{k}(g).
$$
Its inverse map is given by the Fourier inversion formula:
$$f(g) = \frac{1}{|G|}\sum_{{k}=1}^m {\rm Tr}\left(\rho_{k}(g^{-1})\hat f(\rho_{k})\right).$$
Fourier transformation linearises convolution in the following sense. For any $f$ and $g$
in $\mathcal C(G,\mathbf C)$ and any $k=1,\ldots,m$ we have:
$$\widehat{(f\star g)}(\rho_{k}) = \hat f(\rho_{k})\cdot \hat g(\rho_{k}),$$
where the product in the right side of the above equation is a matrix product.

Fourier transformation extends to $\mathcal C(G,{\rm Seq}(\mathbf C))$ in
a natural way by setting $\hat s^{(t)} = \widehat{s^{(t)}}$ for any 
$s\in\mathcal C(G,{\rm Seq}(\mathbf C))$. By construction it commutes
with the shift operator $\sigma$. It maps $\mathcal C(G,{\mathbf k})$ onto the space: 
$$\mathcal O(\widehat{G}, {\mathbf k}) 
= \bigoplus_{{\mathbf k}=1}^m{\rm Mat}(d_\rho\times d_\rho, {\mathbf k}).$$
Thus, the convolution equation \eqref{conv_eq} is transformed
into a system of linear recurrences:
\begin{equation}\label{fou_eq}
\hat s_\rho^{(t+1)} = \hat \alpha_\rho^{(t)}\cdot \hat s_\rho^{(t)}, \quad \rho\in \widehat G, 
\end{equation}
 where the matrices $\hat\alpha(\rho_k)$ are in ${\rm Mat}(d_k\times d_k, \mathbf k)$.
We have thus proven the following result. 
 
\begin{proposition}
Let $G$ be a finite group with irreducible unitary representations $\{\rho_1,\ldots,\rho_m\}$ of
ranks $\{d_1,\ldots,d_m\}$ respectively. Let $\alpha$ be it $\mathcal C(G,\mathbf k)$. Let 
us us consider the convolution equation \eqref{conv_eq} in $G$. 
The equation: 
$$\sigma(\hat s) = \hat \alpha \hat s,$$
obtained by applying the Fourier
transform in $G$ to \eqref{conv_eq} consists of $m$ decoupled linear difference systems of 
rank $d_1,$ $\ldots,$ $d_m$ respectively with coefficients in $\mathbf k$. 
\end{proposition}

The linear homogeneous difference equations of first order that appear for
rank one representations in the Fourier transform of a cellular automaton
are solved in terms of $\mathbf k$-hypergeometric sequences. A 
\emph{$\mathbf k$-hypergeometric} sequence $x^{(t)}\in {\rm Seq}(\mathbf C)$
is a sequence of the form $x^{(t)} = \lambda\prod_{\tau = 0}^{t-1}a^{(\tau)}$ with
$a^{(t)}\in \mathbf k$. This is an ad-hoc definition allowing to solve
all linear homogeneous first order difference equations. The name is
motivated by the fact that the coefficients of the 
classical hypergeometric series are $\mathbf C(t)$-hypergeometric sequences.

\begin{example} \label{ex_cyclic}
Let us consider a generalized linear cellular automaton in a Cayley graph whose
underlying group is the cyclic group $\mathbf Z_T$. The system of linear difference equations 
governing its evolution is written as follows:
\begin{equation}\label{cyclic_ZT}
s_{n}^{(t+1)} = \sum_{m = 0}^{T-1}\alpha_{m}^{(t)} s_{n-m\,(\rm mod\, T)  }^{(t)}, \quad n = 0,\ldots,T-1.
\end{equation}
If we rewrite it in matrix form it yields:
$$\left(\begin{array}{c} s_0^{(t+1)} \\ \vdots \\ s_{T-1}^{(t+1)}  \end{array} \right)
=
A(t)
\left(\begin{array}{c} s_0^{(t)} \\ \vdots \\ s_{T-1}^{(t)}  \end{array} \right),
$$
where $A(t)$ is a circulant matrix of sequences. The matrix element $a_{ij}^{(t)}$ of $A(t)$ is:
$$a_{ij}^{(t)} = \alpha^{(t)}_{i-j\,(\rm mod\, T)}.$$

Let us compute the Fourier transform of equation of \eqref{cyclic_ZT}.  
The dual $\widehat{\mathbf Z_T}$ is the multiplicative group $\mathbf S^1_T\subset \mathbf C^*$ 
of $T$-th roofs of unity. 
To each root $\mu\in \mathbf S^1_T$ it corresponds an
irreducible representation $\rho_{\mu}$:
$$\rho_\mu\colon \mathbf Z_T\to \mathbf C^*, \quad {k} \mapsto \mu^{k}.$$
The product in $\mathbf S_T^1$ is compatible with
the tensor product of representations in the following way:
$\rho_\mu \otimes \rho_{\tau} = \rho_{\mu\tau}$. Note that $\mathbf S^1_T$ is
a cyclic group isomorphic to $\mathbf Z_T$, but the isomorphism depends
of the choice of a primitive $T$-th root of unity. Let us take
$\xi = \exp(T^{-1} 2\pi i)$. For $f\in \mathcal C(\mathbf Z_T,\mathbf C)$ 
we have $\hat f\in\mathcal(\mathbf S^1_T,\mathbf C)$. Let us write $\hat f_{k}$ for
$\hat f(\xi^k)$. The Fourier transform is written in matrix form 
as follows:
$$\left(\begin{array}{c} \hat f_0 \\ \vdots \\ \hat f_{T-1}  \end{array} \right)
=
\left(\begin{array}{ccccc} 
1  & \xi & \xi^2 & \ldots & \xi^{T-1}_T \\
1 & \xi^2 & \xi^4 & \ldots & \xi^{2(T-1)_T}  \\
\vdots & \vdots & \vdots & \ddots & \vdots \\ 
1 & \xi^{T-1} & \xi^{2(T-1)} & \ldots & \xi^{(T-1)(T-1)}  \end{array} \right)
\left(\begin{array}{c} f_0 \\ \vdots \\ f_{T-1}  \end{array} \right)$$
In particular we have:
$$\hat\alpha_{k}^{(t)} = \sum_{n=0}^{T-1} \xi^{nk}\alpha_n^{(t)}.$$
The system for the Fourier transform $\hat s$ of the state $s$ in \eqref{cyclic_ZT} yields,
$$\left(\begin{array}{c} \hat s_0^{(t+1)} \\ \vdots \\ \hat s_{T-1}^{(t+1)}  \end{array} \right)
=
\left(\begin{array}{cccc} 
\hat\alpha^{(t)}_0 &  0 & \ldots & 0 \\
0 & \hat\alpha^{(t)}_1  & \ldots & 0  \\
\vdots & \vdots & \ddots & \vdots \\ 
0 & 0 & \ldots & \hat\alpha_{T-1}^{(t)}  \end{array} \right)
\left(\begin{array}{c} \hat s_0^{(t)} \\ \vdots \\ \hat s_{T-1}^{(t)}  \end{array} \right)$$
Which consist on $T$ decoupled linear homogeneous difference equations. Their solutions can be described 
in terms of $\mathbf k$-hypergeometric sequences:
$$\hat s^{(t)}_m = \hat s_{k}^{(0)}\prod_{\tau = 0}^{t}\hat\alpha_m^{(t)}.$$
The fundamental solution $\phi$ is the solution of \eqref{cyclic_ZT} with initial codition 
$\phi^{(0)}=\delta_0$. The Fourier transform $\widehat{\delta_0}$ 
is the constant function $1\in\mathcal C(\mathbf S^1_T,\mathbf C)$. Thus,
$$\hat\phi_m^{(t)} = \prod_{\tau = 0}^{t}\hat\alpha_m^{(\tau)}.$$
We now apply the inverse Fourier transform to obtain the fundamental solution:
\begin{equation}\label{phi_ZT}
\phi^{(t)}_n = \frac{1}{T}\sum_{k=1}^{T-1}\xi_T^{-kn}\prod_{\tau=1}^t \hat\alpha_{k}^{(\tau)}.
\end{equation}
The $\sigma$-ring extension spanned by all solutions is:
$${\mathbf k} \subseteq 
{\mathbf k}[\beta_1,\ldots,\beta_{T-1}]\subset {\rm Seq(\mathbf C)}, \quad \mbox{ where } \quad 
\beta_m^{(t)} = \prod_{\tau=1}^t \hat\alpha_m^{(\tau)}$$
are $\mathbf k$-hypergeometric sequences.
\end{example}

\begin{example}\label{W90T}
 (Periodic Wolfram's 90) Let us consider the following difference
equations corresponding to an automaton in 
$\mathcal G = (\mathbf Z_T, \{-1,1\})$:
$$s^{(t+1)}_n = s^{(t)}_{n-1 \,({\rm mod}\, T)} + s^{(t)}_{n+1\,({\rm mod}\, T)}, \quad n = 0,\ldots,T-1.$$
When taking states in the finite field $\mathbf F_2$ this is the well-known Wolfram's rule 90 
(see \cite{Voorhees}, page 29). We consider it as a cellular automaton with complex states.
Its fundamental solution $\phi$ can be computed combinatorially. Let us denote $|n|$ the distance
from $n$ to $0$ in $\mathbf Z_T$. For $t<\frac{T-1}{2}$ we have the following expression:
$$\phi_{n}^{(t)} = \begin{cases} 
\left(\begin{array}{c} t \\ \frac{t \pm |n|}{2}\end{array}\right) \quad \mbox{if} \quad  \frac{t \pm |n|}{2}\in 
\left\{1,2,\ldots,t\right\}  
\\  \\
0 \quad\quad\quad\quad\quad \mbox{otherwise.}
\end{cases}.$$
On the other hand, by application of Fourier transform in $\mathbf Z_T$ -this is a 
particular case of the previous example \eqref{cyclic_ZT} - we obtain a formula
for its fundamental solution: 
\begin{equation}\label{phi_w90T}
\phi^{(t)}_n = \frac{2^t}{T}\sum_{m=1}^{T-1} \cos\left(\frac{2\pi m n}{T} \right)\cos^t\left(\frac{2\pi m}{T}\right).
\end{equation}
By comparison of both methods of resolution we obtain a trigonometric development for the combinatorial numbers:
\begin{equation}\label{comb_trig}
\left(\begin{array}{c} t \\ b\end{array}\right) = 
\frac{2^t}{T}\sum_{m=1}^{T-1} \cos\left(\frac{2\pi m (2b+t)}{T} \right)\cos^t\left(\frac{2\pi m}{T}\right),
\end{equation}
that holds for all $T> 2t+1$. Finally, the $\sigma$-ring extension associated to this automaton is:
$$\mathbf C \subset \mathbf C\left[2^t,2^t\cos^t\left(\frac{2\pi}{T}\right),\ldots,
2^t\cos^t\left(\frac{2\pi m}{T}\right),
\ldots,2^t\cos^t\left(\frac{2\pi(\lfloor T/2\rfloor)}{T}\right)
\right].$$ 
\end{example}

\begin{example}
Let us consider a Cayley graph $\mathcal G = (S_3, \{(1,3,2),(1,2)\})$ of fig. \ref{fig3} (right). 
A generalized non-autonomous linear cellular automaton in $\mathcal G$ 
with coefficients in ${\mathbf k}$ is convolution equation \eqref{conv_eq} where the support of 
$\alpha\in\mathcal C(S_3,\mathbf C)$ is contained in $\{(1,3,2),(1,2)\}$. 
Let us denote $\alpha(1,2,3)= a^{(t)}\in {\mathbf k}$ and $\alpha(1,2) = b^{(t)}\in {\mathbf k}$. 
The evolution equations for solutions of the cellular automaton are:
\begin{equation}\label{aut_s3}
\left(\begin{array}{c} s_{(e)}^{(t+1)} \\ s_{(123)}^{(t+1)} \\ s_{(132)}^{(t+1)} \\
s_{(12)}^{(t+1)} \\ s_{(13)}^{(t+1)} \\ s_{(23)}^{(t+1)} \end{array} \right)
=
\left(\begin{array}{cccccc} 
0 & a^{(t)} & 0 & b^{(t)} & 0 & 0 \\
0 & 0 & a^{(t)} & 0 & 0 & b^{(t)}  \\
a^{(t)} & 0 & 0 & 0 & b^{(t)} & 0 \\ 
b^{(t)} & 0 & 0 & 0 & a^{(t)} & 0 \\
0 & 0 & b^{(t)} & 0 & 0 & a^{(t)} \\
0 & b^{(t)} & 0 & a & 0 & 0 
\end{array} \right)
\left(\begin{array}{c}
s_{(e)}^{(t)} \\ s_{(123)}^{(t)} \\ s_{(132)}^{(t)} \\
s_{(12)}^{(t)} \\ s_{(13)}^{(t)} \\ s_{(23)}^{(t)} 
 \end{array} \right)
\end{equation}
The dual $\widehat{S_3} = \{\iota, \epsilon, \varrho \}$, consists of three elements
(see, for instance, \cite{Terras} page 264). Namely, the
trivial representation $\iota_g = 1$, the parity $\epsilon_g = (-1)^{g}$ and the standard
representation, $\varrho_g(e_i)=e_{g(i)}$ defined in the plane $\langle e_1 - e_2, e_1 - e_3 \rangle \simeq \mathbf C^2$.
The Fourier transform of $\alpha$ is:
\begin{eqnarray*}
\hat\alpha \colon   & \iota    \mapsto & a^{(t)} + b^{(t)} \\
                  & \epsilon \mapsto & a^{(t)} - b^{(t)} \\
                  & \varrho  \mapsto & A^{(t)} \colon = \left( \begin{array}{cc} -b^{(t)} & a^{(t)}-b^{(t)} 
                  \\ -a^{(t)} & b^{(t)}-a^{(t)} \end{array}\right) 
\end{eqnarray*}
The Fourier transform of the linear difference system \eqref{aut_s3} consist
of three decoupled linear difference systems, of rank one, one and two respectively:
\begin{eqnarray}
\hat s_\iota^{(t+1)} &=& (a^{(t)} + b^{(t)})\hat s_{\iota}^{(t)}  \label{s1}\\ 
\hat s_\epsilon^{(t+1)} &=& (a^{(t)} - b^{(t)})\hat s_{\epsilon}^{(t)} \label{s2} \\
\left(\begin{array}{cc} \hat s_{\varrho 11}^{(t+1)} & \hat s_{\varrho 12}^{(t+1)} \\ 
\hat s_{\varrho 21}^{(t+1)} & \hat s_{\varrho 22}^{(t+1)} 
\end{array} \right)
&=&
\left(\begin{array}{cc} -b^{(t)} & a^{(t)} - b^{(t)} \\ 
-a_{21}^{(t)} & b^{(t)}-a^{(t)}  
\end{array} \right)
\left(\begin{array}{cc} \hat s_{\varrho 11}^{(t)} & \hat s_{\varrho 12}^{(t)} \\ 
\hat s_{\varrho 21}^{(t)} & \hat s_{\varrho 22}^{(t)} 
\end{array} \right) 
\end{eqnarray}
The $\sigma$-ring extension 
${\mathbf k}\subseteq {\mathbf k}[\phi_{g}^{(t)}]_{g\in S_3}\subset {\rm Seq}(\mathbf C)$ spanned 
by solutions of \eqref{aut_s3} is:
${\mathbf k}[u^{(t)},v^{(t)},w^{(t)}_{11}, w^{(t)}_{12},w^{(t)}_{21},w^{(t)}_{22}]$ where,
$$u^{(t)} = \prod_{\tau=0}^{t-1}(a^{(\tau)} + b^{(\tau)}), 
\quad v^{(t)} = \prod_{\tau=0}^{t-1}(b^{(\tau)}- a^{(\tau)}),$$
$$
\left(\begin{array}{cc} w_{11}^{(t)} & w_{12}^{(t)} \\ 
w_{21}^{(t)} & w_{22}^{(t)} 
\end{array} \right)
=
\left(\begin{array}{cc} -b^{(t-1)} & a^{(t-1)} - b^{(t-1)} \\ 
-a_{21}^{(t-1)} & b^{(t)}-a^{(t-1)}  
\end{array} \right)
\cdots
\left(\begin{array}{cc} -b^{(0)} & a^{(0)} - b^{(0)} \\ 
-a_{21}^{(0)} & b^{(0)}-a^{(0)}  
\end{array} \right).
$$
Note that sequences $u^{(t)}$ and $w^{(t)}$, corresponding to fundamental
$S_3$-periodic and $A_3$-periodic solutions are $\mathbf k$-hypergeometric.
Sequences $w_{ij}^{(t)}$ are not $\mathbf k$-hypergeometric in general. We may
compare graphs in fig. \ref{fig3}. Linear cellular automata in the
left graph are solved by $\mathbf k$-hypergeometric sequences (example \ref{ex_cyclic})
but those in the right graph may have non-  $\mathbf k$-hypergeometric sequences in their general
solution due to the non-commutativity of their underlying group.
\end{example}

\subsection{Discrete abelian groups}
Let us consider $G$ a discrete abelian group. We will use additive notation for the operation
in $G$, that fits with the classical notation of Fourier transform.
All irreducible unitary representations of $G$ are of rank $1$. 
Its dual $\widehat{G}$ is identified with the set of group homomorphisms from $G$ to
the unit circle $\mathbf S^{1}\subset\mathbf C^{*}$.
The set $\widehat G$ is endowed with a product - coming from the tensor product of 
representations or equivalenty from the product in $\mathbf S^1$- that makes it an
abelian group. It is also endowed with the topology coming from the convergence
in compact subsets of $\mathbf G$ of maps from $G$ to $\mathbf S^1$. With these
product and topology $\widehat{G}$ is a compact abelian group (see \cite{Rudin}).
 
As we have seen in example \ref{ex_cyclic}, the dual of the cyclic group $\mathbf Z_T$
is the group $\mathbf S_T^1$ of $T$-th roots of the unity. 
The dual of the integer numbers $\mathbf Z$ is the unit
circle $\mathbf S^1$. Each unitary complex number $e^{i\theta}\in \mathbf S^1$ 
corresponds to an unitary representation:
$$\rho_{\theta} \colon \mathbf Z \to \mathbf S^1\subset \mathbf C^*, \quad n \mapsto e^{in\theta}.$$ 

Let us recall that function $\alpha$ in equation \eqref{conv_eq} has finite support. Thus, we 
consider the Fourier transform for finite support functions, avoiding the problem of convergence
of Fourier series:
$$\mathcal F\colon \mathcal C_c(G,\mathbf C)\hookrightarrow \mathcal C(\widehat{G}, \mathbf C), 
\quad \mathcal F(f)(\rho) = \hat f(\rho) = \sum_{g\in G}f(g)\rho(g).$$
It is useful to introduce the following notation for the image $\mathcal F(\mathcal C_c(G,\mathbf C)$
of the Fourier transform. For each
morphism $\rho\in\widehat{G}$ we write $\theta(\rho) = -i\ln(\rho)$. In this way 
$\theta(\rho)\colon G\to\mathbf R/2\pi\mathbf Z$ is the same homomorphism  $\rho$ but 
expressed in angle coordinates, and for each $g\in G$ the function 
$e^{ig\theta}\colon \widehat{G}\to \mathbf C$ 
maps $\rho$ to $\rho(g)$. 
The image $\mathcal F(\mathcal C_c(G,\mathbf C))$ is thus the ring of \emph{Fourier polynomials}
$\mathbf C[e^{ig\theta}\colon g\in G]$. The Fourier monomials satisfy the relations $e^{ig\theta}e^{ih\theta} = e^{i(g+h)\theta}$, thus $e^{i0\theta} = 1$ and if $g$ is 
a torsion element of order $T$ then $e^{ig\theta}$ is a primitive T-th root of unity. 
The inverse of the Fourier transform is given by the Fourier inversion formula:
$$f(g) = \int_{\widehat G} e^{-ig\theta} \hat f(e^{i\theta}) d\theta,$$
where $d\theta$ denotes the normalized Haar measure in $\widehat G$. Note that:
$$\int_{\widehat G} e^{ig\theta} d\theta = \begin{cases} 
1 \quad \mbox{if} \quad g = 0, \\
0 \quad \mbox{if} \quad g \neq 0.
\end{cases}$$
Now, we take \emph{formal Fourier series} in $\mathbf C[[e^{ig\theta}\colon g\in G]]$. The 
space of formal Fourier series is not a ring. However, the product of a formal Fourier
series and a Fourier polynomial is well defined, and thus the Fourier inversion formula. We have
then an extension of the Fourier transform defined in the space $\mathcal C(G,\mathbf C)$,
$$\mathcal F\colon \mathcal C(G,\mathbf C)\xrightarrow\sim \mathbf C[[e^{ig\theta}\colon g\in G]]
\supset \mathcal C(\widehat{G},\mathbf C),
\quad \mathcal F(f) = \sum_{g\in G}f_g e^{ig\theta}.$$
This Fourier transform swaps product and convolution and maps the convolution equation \eqref{conv_eq}
onto an equation of the form,
$$\sigma(\hat s_\theta) = \hat\alpha_\theta\cdot \hat s_\theta,$$
where $\hat\alpha_\theta$ is a Fourier polynomial and $\hat s_\theta$ a formal Fourier series. For 
the fundamental solution $\phi^{(t)}$ the initial data $\phi_g^{(0)}=\delta_0$ has finite
support and then we easily compute a formula involving $\mathbf k$-hypergeometric sequences 
and integration along $\widehat G$, proving the following result.

\begin{theorem}\label{theorem25}
Let $G$ be a discrete abelian group and $\alpha\in \mathcal C_c(G,\mathbf C)$. The
fundamental solution to the convolution equation \eqref{conv_eq} is:
$$\phi_g^{(t)} = \int_{\widehat G} e^{-ig\theta}\left(\prod_{\tau=0}^{t-1}\hat\alpha_\theta^{(\tau)}\right) d\theta.$$
Hence, the $\sigma$-ring extension $\mathbf k\subseteq L^c = \mathbf k[\phi_g^{(t)}\colon g\in G]$ 
of finite support solutions is spanned by sequences 
obtained by integration along $\widehat G$ of 
$\mathbf k[e^{ig\theta}:g\in G]$-hypergeometric sequences.
\end{theorem}

The action of $G$ on $\mathcal C_c(G,{\rm Seq}(\mathbf C))$ is transported by means of the Fourier
transform to the ring $\rm Seq(\mathbf C)[e^{ig\theta}\colon g\in G]$ of Fourier polynomials
with sequence coefficients. This action is written as $e^{ig\theta}\star h = e^{i(g+h)\theta}$. 
We have the following result:

\begin{theorem}\label{theorem26}
The Fourier transform is a $G$-module isomorphism between 
${\rm Sol}_c(\alpha, L^c)$ and 
$\left(\prod_{\tau=0}^{t-1}\hat\alpha_\theta^{(\tau)}\right)\cdot\mathbf C[e^{ig\theta}\colon g\in G]
\subset {\rm Seq}(\mathbf C)[e^{ig\theta}\colon g\in G]$.
\end{theorem}

\begin{proof}
By proposition \ref{prop_monog}, the space ${\rm Sol}_c(\alpha, L^c)$ 
is the monogenous $G$-module spanned as by the fundamental solution $\phi_g^{(t)}$.
Hence:
$$\mathcal F\left(\oplus_{g\in G}\langle\phi\star g\rangle\right) = 
\oplus_{g\in G}\langle \hat\phi_\theta^{(t)}\cdot e^{ig\theta}\rangle =
\hat\phi_\theta^{(t)}\cdot  C[e^{ig\theta}\colon g\in G].$$
The Fourier transform of $\phi_g^{(t)}$ is 
$\left(\prod_{\tau=0}^{t-1}\hat\alpha_\theta^{(\tau)}\right),$ finishing the proof.
\end{proof}

\begin{example}\label{example_fs}
Let us consider now the following system of linear difference equations:
$$s_{n}^{(t+1)} = \sum_{m = -R}^{R}\alpha_{m}^{(t)} s_{n-m}^{(t)},\quad n\in \mathbf Z,$$
where $\alpha\in\mathcal C(\mathbf Z,{\mathbf k})$ has support in $\{-R,\ldots,R\}$. This is a linear
generalized cellular automaton in the Cayley graph $(\mathbf Z,\{-R,\ldots,R\})$ with coefficients in ${\mathbf k}$. Let us compute the space of solutions with finite support. 
The Fourier transform of $\alpha$ is the 
Fourier polynomial $\hat\alpha\in {\mathbf k}[e^{i\theta},e^{-i\theta}]$,
$$\hat\alpha^{(t)}_\theta = \sum_{m = -R}^R \alpha^{(t)}_m e^{im\theta}.$$
The Fourier transform of the equation 
is a single first order linear difference equation with coefficients
in ${\mathbf k}[e^{i\theta},e^{-i\theta}]$. This equation is solved by a $\mathbf k$-hypergeometric
sequence:
$$\hat s^{(t)} = \hat s^{(0)}\beta^{(t)}\quad \mbox{were}
\quad\beta^{(t)}_\theta = \prod_{\tau = 0}^{t-1}\hat \alpha_\theta^{(\tau)}.$$
By inverse Fourier transform 
we obtain an integral formula for the fundamental solution:
$$\phi_{n}^{(t)}= \frac{1}{2\pi} \int_{0}^{2\pi} e^{-in\theta} \beta^{(t)}_\theta d\theta.
$$
The $\sigma$-extension ${\mathbf k}\subseteq L^c \subset {\rm Seq}(\mathbf C)$ spanned by finite support solutions is:
$$L^c = {\mathbf k}[\ldots,\phi_{-2}^{(t)},\phi_{-1}^{(t)},\phi_0^{(t)},\ldots]$$
And finally, we have an isomorphism of $\mathbf Z$-modules:
$$\mathcal F\colon{\rm Sol}(\alpha, L^c) \xrightarrow\sim 
\beta^{(t)}_\theta\cdot \mathbf C[e^{i\theta},e^{-i\theta}],
\quad \phi^{(t)}_n\star m \mapsto \beta_\theta^{(t)} e^{im\theta}.$$
\end{example}

\begin{example}\label{ex_W90Z}
(Finite support Wolfram's 90) The following system is a particular case of 
example \ref{example_fs} analogous to example \ref{W90T}:
\begin{equation}\label{W90Z}
s^{(t+1)}_n = s^{(t)}_{n-1} + s^{(t)}_{n+1}, \quad n\in\mathbf Z
\end{equation}
The fundamental solution is computed combinatorially:
$$\phi_{n}^{(t)} = \begin{cases} 
\left(\begin{array}{c} t \\ \frac{t - n}{2}\end{array}\right) \quad \mbox{if} \quad  \frac{t-n}{2}\in 
\left\{1,2,\ldots,t\right\}  
\\  \\
0 \quad\quad\quad\quad\quad \mbox{otherwise.}
\end{cases}.$$
On the other hand, by application of Fourier transform in $\mathbf Z$ we obtain another formula
for the fundamental solution. Note that in this case $\hat\alpha_\theta = 2\cos(\theta)$
\begin{equation}\label{phi_w90Z}
\phi^{(t)}_n = \frac{2^{t-1}}{\pi}\int_0^{2\pi} \cos(n\theta)cos^t(\theta)d\theta.
\end{equation}
Thus, by comparison of both methods of resolution we obtain a trigonometric development for the combinatorial 
numbers:
\begin{equation}\label{comb_int}
\left(\begin{array}{c} t \\ b\end{array}\right) = 
\frac{2^{t-1}}{\pi}\int_0^{2\pi} \cos((2b+t)\theta)\cos^t(\theta)d\theta.
\end{equation}
Compare the above formulas with equations \eqref{phi_w90T} and \eqref{comb_trig}. Finally the
$\sigma$-ring of sequences spanned by the finite support solutions is:
$$L^c = \mathbf C[\ldots,\phi_{-1}^{(t)},\phi_0^{(t)},\phi_1^{(t)},\ldots],$$
and the Fourier transform:
$$\mathcal F\colon {\rm Sol}_c(\alpha,L^c) \xrightarrow{\sim} 
2^t\cos^t(\theta)\cdot\mathbf C[e^{i\theta},e^{-i\theta}].$$
\end{example}

Finally let us compute all periodic solutions of a linear cellular automaton in a discrete
abelian group. Let us denote by ${\rm Tor}(\widehat G)$ the torsion part of $G$, that is
the subgroup of $\widehat{G}$ consisting of finite order elements. We have the following result:

\begin{theorem}\label{theorem29}
Let $G$ be discrete abelian group, $\alpha$ in $\mathcal C_c(G,\mathbf C)$ and let us consider
the convolution equation \eqref{conv_eq}. The $\sigma$-ring extension spanned by all periodic
solutions is:
$$\mathbf k \subseteq L^{\rm per} =
\mathbf k\left[ \prod_{\tau=0}^{t-1} \hat\alpha_{\theta}^{(\tau)} \colon \theta\in 
{\rm Tor}(\widehat{G}) \right].$$  
\end{theorem}

\begin{proof}
By definition $L^{\rm per} = \bigcup_{H\subseteq G} L^{H-\rm per}$ where $H$ runs over the set
of finite index subgroups of $G$. Let us compute $L^{H-\rm per}$ for any $H\subseteq G$ of 
finite index. The quotient map $\pi\colon G\mapsto G/H$ induces an embedding 
$\hat\pi \colon \widehat{G/H}\hookrightarrow \widehat G$, and $\widehat{G/H}$ is identified
with the finite subgroup of the compact group $\widehat G$ consisting of group morphisms 
$\theta\colon G\to \mathbf S^1$ whose kernel contains $H$. We have the following 
commutative diagram relating the Fourier transforms in $G$ and $G/H$:
$$\xymatrix{
\mathcal C_c(G,\mathbf C) \ar[rr]^-{\mathcal F_G} & &
\mathcal C(\widehat G,\mathbf C) \ar[d]^-{\hat \pi_H^*} \\
\mathcal C(G/H, \mathbf C) \ar[rr]_-{\mathcal F_{G/H}}^-\sim \ar[u]^-{\pi_H^*} 
& &\mathcal C(\widehat{G/H},\mathbf C )
}$$
The map $\hat\pi^*_H$ is simply the restriction of functions.
In the notation of the angle variables $\theta$, we have 
$\hat\pi_H^*(e^{ig\theta}) = e^{i[g]\theta}$.

From lemma \ref{adjoint} we have that the fundamental $H$-periodic
solution $(\phi^H)_g^{(t)}$ is given by:
$$(\phi^H)_g^{(t)} = 
\pi_H^*\mathcal F_{G/H}^{-1}\left(\prod_{\tau=0}^{t-1}\widehat{r_H\alpha}_\theta^{(t)}\right).$$
The map $\pi^*_H$ is the extension of functions from $G/H$ to $G$ and $F_{G/H}$ is a
$\rm Seq(\mathbf C)$-linear isomorphism. Thus,
$$L^{H-\rm per} = \mathbf k
\left[ \prod_{\tau=0}^{t-1}\widehat{r_H\alpha}_\theta^{(\tau)}\colon \theta \in \widehat{G/H} \right].$$
It is elementary to check that $\mathcal F_{G/H}^{-1}\circ \hat \pi_H^* \circ \mathcal F_G = r_H$.
If follows $\widehat{r_H\alpha}_\theta^{(t)} = \widehat{\alpha}^{(t)}_{\hat\pi_H(\theta)}$ and then
$$L^{H-\rm per} = \mathbf k
\left[ \prod_{\tau=0}^{t-1}\widehat{\alpha}_\theta^{(\tau)}\colon H\subseteq \ker(\theta)\right].$$
Now, a representation $\theta\in \widehat G$ is of finite order if and only if its
kernel is of finite index. Thus, we have the stated result.
\end{proof}

\begin{example} Let us consider the cellular automaton of example \ref{example_fs}. The $\sigma$-ring
extension spanned by the fundamental $\mathbf Z_T$-periodic solution is that of example \ref{ex_cyclic}.    The $\sigma$-ring extension spanned by all periodic solutions is:
$$L^{\rm per} = \mathbf k\left[\beta_ \theta^{(t)} \colon \theta/\pi \in \mathbf Q\cap[0,1) \right].$$
In the particular case of example \ref{ex_W90Z} the $\sigma$-extension spanned by all periodic solutions
is:
$$L^{\rm per} = \mathbf k\left[2^t\cos^t(\pi q) \colon q \in \mathbf Q\cap[0,1) \right].$$
\end{example}

\section{Difference Galois theoretic interpretation}

Difference Galois theory deals with algebraic properties of difference equations. In this context,
Picard-Vessiot-Franke (PVF) theory deals with systems of linear difference equations,
\begin{equation}\label{lin_rec}
x^{(t+1)} = A^{(t)}x^{(t)}, \quad A\in {\rm Mat}(d\times d,\mathbf k).
\end{equation}
In order to apply PVF theory we need to consider a $\sigma$-field of coefficients,
and $\sigma$ should be a field automorphism. In order to make the shift operator $\sigma$
into an automorphism the first step we do is to replace the $\sigma$-ring ${\rm Seq}(\mathbf C)$
of sequences by the $\sigma$-ring ${\rm Seq}_\infty(\mathbf C)$ of germs at infinity of sequences.
That means that our considerations will only related to the eventual behaviour of sequences.
\emph{From now on, $\mathbf k$ is a sub-$\sigma$-field of ${\rm Seq}_\infty(\mathbf C)$
containing the complex numbers}. The general PVF theory here summarized is
exposed with detail \cite{Si_etal}, chapter 1 and 2.

\subsection{PVF extensions and Galois correspondence}

The difference Galois theory assigns to each linear difference system a $\sigma$-ring
containing its solutions.

\begin{definition}
A $\sigma$-ring extension $\mathbf k \subset \mathbf L$ is called a PVF extension for
the linear system if it satisfies:
\begin{itemize}
\item[(a)] $\mathbf L$ is a simple $\sigma$-ring.
\item[(b)] The rank over the constants of the space of solutions of \eqref{lin_rec} in $\mathbf L^d$ equal to the rank of the matrix $A$. 
\item[(c)] $\mathbf L$ does not contain any proper $\sigma$-ring satisfying (a) and (b).
\end{itemize}
\end{definition}

\begin{remark} Condition (a) forces the PVF extension to be bigger that the extension 
spanned by the solutions. In order to have a simple $\sigma$-ring we need to add
the inverse of the determinant of a fundamental
matrix of solutions.
\end{remark} 

The PVF extension is unique up to $\sigma$-isomorphism, and it does not have new constants. The group of
$\sigma$-automorphism $\rm Aut_\sigma(\mathbf L/\mathbf k)$ is naturally endowed with a structure
of linear $\mathbf C$-algebraic group. In order to give a Galois correspondence 
we need to replace $\mathbf L$ by its total ring of fractions $\rm q\mathbf L$. The group of
automorphisms ${\rm Aut}_\sigma({\rm q}\mathbf L / \mathbf k)$ coincides with 
${\rm Aut}_\sigma({\rm q}\mathbf L / \mathbf k)$. The following result quotes
Theorem 1.29 and Corollary 1.30 in \cite{Si_etal}.

\begin{theorem}\label{Gal_corr}
Let $\mathfrak F$ denote de set of intermediate $\sigma$-rings $
\mathbf k\subseteq\mathbf F \subseteq \rm q\mathbf L$ such that any non-zero
divisor of $\mathbf F$ is a unit of $\mathbf F$. Let $\mathfrak G$ denote
the set of algebraic subgroups of ${\rm Aut}_\sigma({\rm q}\mathbf L / \mathbf k)$.
\begin{itemize}
\item[(a)] For any $\mathbf F\in\mathfrak F$ the subgroup 
${\rm Aut}_\sigma({\rm q}\mathbf L / \mathbf F)$
of $\sigma$-automorphisms that fix $\mathbf F$ pointwise, is an algebraic subgroup
of  ${\rm Aut}_\sigma({\rm q}\mathbf L / \mathbf k)$.
\item[(b)] For any algebraic subgroup $H\subseteq {\rm Aut}_\sigma({\rm q}\mathbf L / \mathbf k)$
the fixed $\sigma$-ring ${\rm q}\mathbf L^H$ is in $\mathfrak F$.
\item[(c)] Let $\alpha\colon \mathfrak F\to \mathfrak G$ and $\beta\colon \mathfrak G\to\mathfrak L$
denote de maps $\mathbf F\mapsto {\rm Aut}_\sigma({\rm q}\mathbf L / \mathbf F)$ and
$H\mapsto {\rm q}\mathbf L^H= \{a\in  {\rm q}\mathbf L \colon \forall\tau\in H\,\tau(a)=a\}$.  Then $\alpha$ and $\beta$ are each other's inverses.
\item[(d)] The group $H\in\mathfrak G$ is a normal subgroup if and only if 
$\mathbf F^{{\rm Aut}_\sigma(\mathbf F/\mathbf k)} = \mathbf k$ where $\mathbf F$ is $\rm q\mathbf L^H$.
In such case 
${\rm Aut}_\sigma(\mathbf F/\mathbf k)\simeq {\rm Aut}_\sigma({\rm q}\mathbf L/\mathbf k)/H$.
\end{itemize}
\end{theorem}

\subsection{Splitting}\label{splitting}
Now we will give a geometric construction for the Galois group. In this section we
assume that the matrix of coefficients $A^{(t)}$ in \eqref{lin_rec} is non degenerated. 
Thus, we can take a fundamental matrix of solutions -this can be done by taking
representatives of the sequences in ${\rm Seq}(\mathbf C)$ and the identity
matrix as initial condition- $U\subset{\rm GL}(d,{\rm Seq}_\infty(\mathbf C))$.  
Then the $\sigma$-ring $\mathbf L = \mathbf k[U,U^{-1}]$ 
spanned by the matrix elements of $U$ and its inverse is a PVF extension of $\mathbf k$
for the equation \eqref{lin_rec}. In terms of this matrix $U$ we can construct a splitting
morphism:
$${\rm Split}\colon
\mathbf L \otimes_{\mathbf C}  \mathbf C[\rm{GL}(d,\mathbf C)]  
\to \mathbf L\otimes_{\mathbf k} \mathbf L,
\quad 1\otimes \tau \mapsto 1 \otimes U\tau, \quad U \otimes 1 \mapsto U \otimes 1.$$
The kernel ideal of the splitting morphism turns out to be spanned by its 
restriction $\mathcal G$ to $\mathbf C[\rm{GL}(d,\mathbf C)]$. The variety of $\mathcal G$
is a realization of ${\rm Aut}_\sigma(\mathbf L/\mathbf k)$ as an algebraic subgroup of
${\rm GL}(d,\mathbf C)$. The splitting morphism descends thus to an isomorphism: 
$${\rm split}\colon \mathbf L \otimes_{\mathbf C}
\mathbf C[{\rm Aut}_\sigma(\mathbf L/\mathbf k)] 
 \xrightarrow{\sim} \mathbf L\otimes_{\mathbf k} \mathbf L.$$
This ${\rm split}$ isomorphism maps the Hopf algebra $\mathbf C[{\rm Aut}_\sigma(\mathbf L/\mathbf k)]$
to the ring of constants $(\mathbf L \otimes_{\mathbf k} \mathbf L)^{\sigma}$.

\subsection{Gauge transformations} If we realize a linear change of variables
in the equation \eqref{lin_rec} of the form,
$$y^{(t)} = B^{(t)}x^{(t)}, \quad B\in {\rm GL}(d,\mathbf k),$$ 
we obtain a new linear difference system:
\begin{equation}\label{reduced}
y^{(t+1)} =  C^{(t)}y^{(t)}, \quad C^{(t)} = B^{(t+1)}A^{(t)}B^{(t) -1}.
\end{equation}
It is clear that the PVF extension and Galois groups for \eqref{lin_rec} and \eqref{reduced}
coincide. Thus it is interesting to find the linear change of variables that reduces
the linear difference system to its simplest form. We have the following result:

\begin{theorem}
Let us assume that $H\subset{\rm GL}(d,\mathbf C)$ is an algebraic subgroup such that the matrix
of coefficient $A^{(t)}$ of equation \eqref{lin_rec} is in $H\otimes_{\mathbf C}\mathbf k$. 
Then there is a fundamental
matrix of solutions of  \eqref{lin_rec} in $H\otimes_{\mathbf C} \mathbf L$ 
and the splitting morphism realizes
the Galois group of \eqref{lin_rec} as an algebraic subroup of $H$.
\end{theorem}

Thus, whenever a linear change of variables maps a linear difference system \eqref{lin_rec}
to another \eqref{reduced} in which the matrix of coefficients takes values in a smaller
algebraic subgroup $H\subset {\rm GL}(d,\mathbf C)$, we say that \eqref{reduced} is a 
reduced form of \eqref{lin_rec} to $H$.

\subsection{Cellular automata in finite Cayley graphs} The following results summarizes
what Fourier transforms tell us about the Galois groups of generalized
non-autonomous lineal cellular in finite Cayley graphs.

\begin{theorem}\label{theorem35}
Let $G$ be a finite group, and 
$\alpha \in \mathcal C(G,\mathbf k)$. Let 
$\mathbf k \subseteq \mathbf L$
be the PVF extension of the convolution equation \eqref{conv_eq}. Let
$\widehat G = \{\rho_1,\ldots,\rho_m\}$ where ${\rm rank}(\rho_i) = d_i$. 
The following sentences hold:
\begin{itemize}
\item[(a)] Assume that the map 
$s\mapsto \alpha\star s$ is injective in 
$\mathcal C(G,\mathbf k)$. Then the Fourier transform of \eqref{conv_eq}
is a reduced form to the group 
$\prod_{i = 1}^m {\rm GL}(d_i,\mathbf C).$
\item[(b)] ${\rm dim}_{\mathbf C}({\rm Aut}_\sigma(\mathbf L/\mathbf k)) \leq |G|$.  
\end{itemize}
\end{theorem}

\begin{proof}
(a) The Fourier transform of \eqref{conv_eq} consist of $m$ linear difference systems
of rank $d_1,\ldots,d_m$. If the map $s\mapsto \alpha\star s$ is injective then $\hat\alpha^{(t)}_{\rho_i}$
is non degenerated for each $i = 1,\ldots,m$, and thus the Fourier transform of \eqref{conv_eq} is
a reduced form as stated. (b) As before, the Fourier transform consist of $m$ decoupled
linear difference systems. After reducing them to non degenerate form, be obtain systems of 
rang at most $d_1,\ldots,d_n$. Thus, the Galois group has dimension at most $d_1^2+\ldots d_n^2 = |G|$.
\end{proof}

\begin{example}
Let us consider Wolfram 90 cellular automaton \eqref{W90T} in the cyclic group $\mathbf Z_6$.
The Fourier transform in $\mathbf Z_6$ of the equation is,
\begin{equation}\label{P6eq}
\left(\begin{array}{c} \hat s_{0}^{(t+1)} \\ \hat s_{1}^{(t+1)} \\ 
\hat s_{2}^{(t+1)} \\
\hat s_{3}^{(t+1)} \\ \hat s_{4}^{(t+1)} \\ \hat s_{5}^{(t+1)} \end{array} \right)
=
\left(\begin{array}{cccccc} 
2 & 0 & 0 & 0 & 0 & 0 \\
0 & 1 & 0 & 0 & 0 & 0  \\
0 & 0 & -1 & 0 & 0 & 0 \\ 
0 & 0 & 0 & -2 & 0 & 0 \\
0 & 0 & 0 & 0 & -1 & 0 \\
0 & 0 & 0 & 0 & 0 & -1 
\end{array} \right)
\left(\begin{array}{c} \hat s_{0}^{(t)} \\ \hat s_{1}^{(t)} \\ 
\hat s_{2}^{(t)} \\
\hat s_{3}^{(t)} \\ \hat s_{4}^{(t)} \\ \hat s_{5}^{(t)} \end{array} \right)
\end{equation}
Its PVF-extension is $\mathbf C\subset \mathbf C[2^t,(-1)^t]$.
The difference Galois group of \eqref{P6eq} is the group
of diagonal matrices satisfying the same multiplicative resonances that the
diagonal elements of the the matrix of coefficients, 
see \cite{Si_etal} 2.2. Thus:
$${\rm Gal}(\mathbf C,\phi^6) = 
\left\{\left.
\left(
\begin{array}{cccccc} 
\lambda_0 & 0 & 0 & 0 & 0 & 0 \\
0 & 1 & 0 & 0 & 0 & 0  \\
0 & 0 & \lambda_2 & 0 & 0 & 0 \\ 
0 & 0 & 0 & \lambda_3 & 0 & 0 \\
0 & 0 & 0 & 0 & \lambda_4 & 0 \\
0 & 0 & 0 & 0 & 0 & 1 
\end{array} \right)\right| 
\begin{array}{ccc} 
\lambda_0\lambda_3^{-1}&=& 1 \\
\lambda_2^2 &=& 1 \\
\lambda_2\lambda_4^{-1}&=&1 \\ 
\lambda_0^{-1}\lambda_2\lambda_3 &=& 1 
\end{array} 
\right\}
$$
It is not difficult to check that it is isomorphic to $\mathbf C^*\times\mathbf S^1_2$
which is the group of $\sigma$-automorphisms of $\mathbf C[2^t,(-1)^t]$.

\end{example}

\subsection{Cellular automata in infinite Cayley graphs}
Here we show the difference Galois structure of the periodic solutions. 
We consider a discrete infinite group $G$, and $\alpha \in \mathcal C_c(G,\mathbf k)$.

For each $H\lhd G$ of finite index we consider the $\sigma$-extension 
$\mathbf L^{H-\rm per} = \mathbf k[(\phi^H)_g^{(t)}]_{g\in G}$ spanned by the germ
at $t=\infty$ of the $H$-periodic fundamental solution $\phi^H$ of \eqref{conv_eq}.
This extension $\mathbf k\subset \mathbf L^{H-\rm per}$ is in fact a PVF extension
of the finite dimensional linear difference system:
$$\sigma(s) = r_H\alpha \star s$$
defined in $\mathcal C(G/H,\mathbf k)$. The $\sigma$-extension spanned by all periodic
solutions is:
$$\mathbf L^{\rm per} = \lim_{\substack{\longrightarrow \\ H} } \mathbf L^{H-\rm per}, 
\quad\quad \mbox{for } H\lhd G \mbox{ and } |G/H| < \infty.$$

\begin{lemma}
Let $H_1\subset H_2\lhd G$ with $H_1\lhd G$ of finite index. The natural inclusion 
$\mathbf L^{H_2-\rm per} \subseteq \mathbf L^{H_1-\rm per}$ induces a
canonical exact sequence of algebraic groups:
$$\{e\} \to {\rm Aut}_\sigma({\rm q}\mathbf L^{H_1-\rm per}/ \mathbf {\rm q}L^{H_2-\rm per}) \to
{\rm Aut}_\sigma(\mathbf L^{H_1-\rm per}/ \mathbf k) \to
{\rm Aut}_\sigma(\mathbf L^{H_2-\rm per}/ \mathbf k) \to
\{e\}.$$
\end{lemma}

\begin{proof}
By the Galois correspondence theorem \ref{Gal_corr} 
${\rm Aut}_\sigma({\rm q}\mathbf L^{H_1-\rm per}/ \mathbf {\rm q}L^{H_2-\rm per})$
is an algebraic subgroup of ${\rm Aut}_\sigma(\mathbf L^{H_1-\rm per}/ \mathbf k)$ 
whose $\sigma$-ring of fixed elements is $\mathbf {\rm q}L^{H_2-\rm per}$ by
definition. Now, $\mathbf {\rm q}L^{H_2-\rm per}$ is the total ring of fractions
of a PVF extension of $\mathbf k$ and thus it satisfies condition (d) in 
theorem \ref{Gal_corr}, which gives us the exact sequence of the statement.
\end{proof}

  Thus, the set:
  $$\{{\rm Aut}_\sigma(\mathbf L^{H-\rm per}/ \mathbf k) \colon  H\lhd G,\, |G/H|\leq \infty\}$$
is a projective system of algebraic groups. We define \emph{the Galois group} of 
$\mathbf L^{\rm per}/\mathbf k$ as the projective limit:
  $${\rm Gal}(\mathbf L^{\rm per}/\mathbf k) =  
  \lim_{\substack{\longleftarrow \\ H} } {\rm Aut}_\sigma(\mathbf L^{H-\rm per}/ \mathbf k), 
\quad\quad \mbox{for } H\lhd G \mbox{ and } |G/H| < \infty.$$

  Each element $\tau\in {\rm Gal}(\mathbf L^{\rm per}/\mathbf k)$ is then a chain of compatible
$\sigma$-ring automorphisms of the extensions $\mathbf k\subseteq \mathbf L^{H-\rm per}$, 
and thus $\tau$ defines a $\sigma$-ring automorphism of the union 
$\mathbf k \subseteq \mathbf L^{\rm per}$. We have thus a natural inclusion 
${\rm Gal}(\mathbf L^{\rm per}/\mathbf k) \subseteq {\rm Aut}_{\sigma}(\mathbf L^{\rm per}/\mathbf k)$.

\begin{proposition}\label{pro_Gal}
${\rm Gal}(\mathbf L^{\rm per}/\mathbf k) = {\rm Aut}_{\sigma}(\mathbf L^{\rm per}/\mathbf k)$.
\end{proposition}

\begin{proof}
We have to see that all automorphisms of $\mathbf L^{\rm per}/\mathbf k$ come from a compatible
family of automorphisms $\{\tau_H\}$ with $\tau_H\in {\rm Aut}_{\sigma}(\mathbf L^{H-\rm per}/\mathbf k)$.
In other words, we have to check that for each $H\lhd G$ of finite index $\tau(\mathbf L^{H-\rm per}) = 
\mathbf L^{H-\rm per}$. First we extend $\tau$ to an automorphism of the total ring of
fractions. Second, let $\psi\in\mathcal C(G,\mathbf L)$ be the germ at infinity of a fundamental 
$H$-periodic solution of equation
\eqref{conv_eq} - we may have to take some representatives of the coefficients -.
Then ${\rm q}\mathbf L^{H-\rm per} = {\rm q}\mathbf k[\psi^{(t)}_g\colon g\in G]$.
Then, since $\tau$ is a $\sigma$-ring automorphism, $\tau\psi$ is also 
$H$-periodic solution of 
\eqref{conv_eq} - satisfying the same no degeneracy conditions in order to span the whole
space of solutions - and thus 
$\tau({\rm q}\mathbf L^{H-\rm per}) = $ 
${\rm q}\mathbf k[\tau \psi^{(t)}_g\colon g\in G]= {\rm q}L^{\mathbf H-\rm per}$.
\end{proof}

The group ${\rm Gal}(\mathbf L^{\rm per}/\mathbf k)$ is an affine pro-algebraic group. Its ring
of regular functions is the direct limit of the rings of its algebraic quotients:
$$\mathbf C[{\rm Gal}(\mathbf L^{\rm per}/\mathbf k)] =  
  \lim_{\substack{\longrightarrow \\ H} } \mathbf C[{\rm Aut}_\sigma(\mathbf L^{H-\rm per}/ \mathbf k)], 
\quad\quad \mbox{for } H\lhd G \mbox{ and } |G/H| < \infty.$$

It is clear that $\mathbf k\subseteq \mathbf L^{\rm per}$ is not a PVF extension. There are some 
broader extensions of Galois in which this particular case fits. An exposition of a very 
general theory that applies to our case is given in \cite{Am_etal}.

\begin{theorem}\label{theorem39}
The $\sigma$-ring extension $\mathbf k\subseteq\mathbf L^{\rm per}$ is a Hopf-Galois in the
following sense. There is a $\sigma$-ring isomorphism:
$${\rm split}\colon \mathbf L^{\rm per} \otimes_{\mathbf C}
\mathbf C[{\rm Gal}(\mathbf L^{\rm per}/\mathbf k)] 
 \xrightarrow{\sim} \mathbf L^{\rm per}\otimes_{\mathbf k} \mathbf L^{\rm per}.$$
that identifies $\mathbf C[{\rm Gal}(\mathbf L^{\rm per}/\mathbf k)]$ with the ring of constants
$(\mathbf L^{\rm per}\otimes_{\mathbf k}\mathbf L^{\rm per})^\sigma$.   
\end{theorem}

\begin{proof}
Let $H$ and $F$ be normal subgroups of $G$ of finite index, with $H \subset F$
we have a projection $\pi_{H,F}\colon {\rm Aut}_\sigma(\mathbf L^{H-\rm per}/\mathbf k)
\to {\rm Aut}_\sigma(\mathbf L^{F-\rm per}/\mathbf k)$, and thus an embedding
of Hopf algebras:
$$\pi_{H,F}^*\colon \mathbf C[{\rm Aut}_\sigma(\mathbf L^{F-\rm per}/\mathbf k)]
\hookrightarrow \mathbf C[{\rm Aut}_\sigma(\mathbf L^{H-\rm per}/\mathbf k)],$$
given by the Galois correspondence. This construction is compatible with that of
the splitting morphisms of subsection \ref{splitting}, thus we have a commutative
diagram:
$$\xymatrix{
 \mathbf L^{H-\rm per} \otimes_{\mathbf C} \mathbf C[{\rm Aut}_\sigma(\mathbf L^{H-\rm per}/\mathbf k)] 
 \ar[rr]_-{\sim}^-{{\rm split}_H}
 & & 
 \mathbf L^{H-\rm per}\otimes_{\mathbf k} \mathbf L^{H-\rm per} \\
 \mathbf L^{F-\rm per} \otimes_{\mathbf C} \mathbf C[{\rm Aut}_\sigma(\mathbf L^{F-\rm per}/\mathbf k)]
  \ar[rr]_-{\sim}^-{{\rm split}_F} \ar[u]^-{i\otimes \pi^{*}_{H,F}}
 & &
 \mathbf L^{F-\rm per}\otimes_{\mathbf k} \mathbf L^{F-\rm per} \ar[u]^-{i\otimes i}
}$$
Thus we have a directed system of isomorphism, that induces an isomorphism between the direct
limits. The direct limit commutes with tensor products, so that, this is the isomorphism of
the statement. Finally, from proposition \ref{pro_Gal} we have that
the constants of the direct limit are also the direct limit of the constants.
\end{proof}

\begin{example}
Let us assume that the group $G$ is abelian. Then the construction of the Galois group 
${\rm Gal}(\mathbf L^{\rm per}/\mathbf k)$ is rather simpler. The $\sigma$-ring is
constructed by adjoining to $\mathbf k$ - in a compatible way - the solutions
of all the difference equations:
$$\hat s^{t+1}_{\theta} = \hat\alpha_{\theta}^{(t)}\hat s^{(t)}_{\theta},
\quad \theta\in{\rm Tor}(\widehat G)$$ 
Taking a representative of $\alpha$ in $\mathcal C(G,{\rm Seq}(\mathbf C)$ we solve
those difference equations by means of $\mathbf  k$-hypergeometric sequences whose
germs at $t=\infty$ we denote by $\psi_\theta^{(t)}$. Then,
 $$\mathbf L^{\rm per} = \mathbf k\left[\psi_{\theta}^{(t)}, \frac{1}{\psi_{\theta}^{(t)}}
 \colon
 \theta\in {\rm Tor}(\widehat G) \mbox{ and } \hat\alpha_{\theta}\neq 0
 \right]$$
Here, the $\mathbf  k$-hypergeometric $\psi^{(t)}_\theta$ with $\hat\alpha_\theta = 0$
are eventually zero and thus not included. Now we can see the Galois group as a 
pro-algebraic subgroup of the infinite torus:
$$\mathbf T_G = \prod_{\theta\in {\rm Tor}(\widehat G)} \mathbf C^*_\theta$$
consisting the product of a copy of a multiplicative group $\mathbf C^*$ for each element
of ${\rm Tor}(\widehat G)$. For a given element $\lambda\in\mathbf T_G$ we denote
$\lambda_\theta$ its $\theta$-th component. Thus:
$$\lambda(\psi_\theta^{(t)}) = \lambda_\theta \psi_{\theta}^{(t)}.$$
The equations of ${\rm Gal}(\mathbf L^{\rm per}/\mathbf k)$ inside $\mathbf T_G$ are then:
$$\lambda_{\theta} = 1 \mbox{ for any }\theta\mbox{ with }\hat\alpha_\theta = 0,$$
and the rest of equations come from the finite dimensional diagonal equations that appear
for ech $H$ of finite order in $G$. Thus, following \cite{Si_etal} it appears an equation of the
form:
$$\lambda_{\theta_1}^{m_1}\cdot\ldots\cdot \lambda_{\theta_k}^{m_k} = 1$$
if there exist a $k$-tuple $(f_1^{(t)}\ldots,f_k^{(t)})\in\mathbf k^k$ with
$$\left(\frac{f_1^{(t+1)}\hat\alpha_{\theta_1}}{f_1^{(t)}}\right)^{m_1}\cdot \ldots\cdot 
\left(\frac{f_k^{(t+1)}\hat\alpha_{\theta_k}}{f_k^{(t)}}\right)^{m_k} = 1.
$$
\end{example}

\begin{example}
Let us consider Wolfram 90 cellular automaton \eqref{W90Z} where $\mathbf C$ stands
for the base $\sigma$-field. The PVF extension of periodic solutions is: 

$$\mathbf C\subset \mathbf L^{\rm per} = \mathbf C \left[2^t\cos^t(2\pi q), 2^{-t}\cos^{-t}(2\pi q):
q\in \mathbf Q\cap \left[0,{1}\right) \mbox{ and } q \neq \frac{1}{4},\frac{3}{4} \right]$$
Let us discuss the Galois group associated to this extension. We consider the Fourier
transform of \eqref{W90Z} obtaining the family of diagonal equations:
$$\hat s^{(t+1)}_q = 2\cos(2\pi q) \hat s_{q},\quad\quad q\in \mathbf Q\cap[0,1).$$
We consider the embedding of the Galois group in the infinite torus: 
$${\rm Gal}_\sigma (\mathbf L^{\rm per}/\mathbf C) \subset  \prod_{q\in \mathbf Q\cap[0,1)}\mathbf C^{*}.$$
A tuple $(\lambda_q)_{q\in \mathbf Q\cap[0,1)}$ of the infinite torus 
is in the Galois group if and only if it
satisfies $\lambda_q=1$ for the zeroes of $\cos(2\pi q)$ and equations:
$$\lambda_{q_1}^{m_1}\cdot\ldots\cdot\lambda_{q_r}^{m_r} =1,$$
for all resonances of the form:
$$2^{m_1+\ldots+m_r}\cos^{m_1}(2\pi q_1)\cdot\ldots\cdot\cos^{m_r}(2\pi q_r) = 1.$$
Some of these equations, of low degree, can be computed explicitely:
$$\lambda{\frac{1}{4}}=1, \quad \lambda_{\frac{1}{6}}=1, \quad \lambda_{\frac{1}{3}}^2=1,$$
$$\lambda_q\lambda_{1-q}^{-1}=1, \quad\lambda_{\frac{1}{3}}\lambda_q\lambda^{-1}_{q+\frac{1}{2}} = 1,\quad
\mbox{ for any } 
q\in \left[0,\frac{1}{2}\right)\cap \mathbf Q.$$
\end{example}

\begin{remark}
In this article we succeed in studying the difference Galois structure of the $\sigma$-ring
spanned by the periodic solutions of a cellular automata. It would be desirable to apply
the difference Galois theory to the $\sigma$-ring spanned by finite support solutions.
In the abelian case, the ring of Fourier polynomials is endowed of a differential structure
that allows to apply parameterized difference Galois theory (see \cite{Ha_etal} for the
general theory and \cite{Wibmer} to a suitable extension).
\end{remark}

\section*{Acknoledgements}

We would like to thank to G. Casale, M. F. Singer and M. Wibmer for their comments and 
discussions on the subject of the article. 
The authors received partial support from the projects MTM2012-31714 from Spanish goverment, 
ECOS Nord France-Colombia No C12M01, and their hosting institutions 
Universidad Sergio Arboleda and Universidad Nacional de Colombia.

\end{document}